\documentclass{amsart}
\usepackage[hmargin=3.5cm,vmargin=3.5cm]{geometry}

\usepackage[utf8]{inputenc}

\usepackage{amsmath,amssymb,amsthm, mathtools}

\usepackage{graphicx}

\usepackage{enumerate,blindtext}
\usepackage[linktocpage=true]{hyperref}

\usepackage{todonotes, lipsum}
\usepackage[all,cmtip]{xy}

\makeatletter
\g@addto@macro\bfseries{\boldmath}
\makeatother

\hypersetup{
bookmarksnumbered=true
}

\newtheorem{theorem}{Theorem}
\newtheorem{lemma}{Lemma}[section]
\newtheorem{korollar}[lemma]{Corollary}

\newtheorem{proposition}[lemma]{Proposition}
\theoremstyle{definition}

\newtheorem{remark}[lemma]{Remark}
\newtheorem{definition}[lemma]{Definition}

\newtheorem*{notation}{Notation}
\newtheorem*{korollar*}{Corollary}

\newcommand{\G}{\ensuremath{\Gamma}}
\newcommand{\g}{\ensuremath{\gamma}}

\newcommand{\C}{\ensuremath{\mathbb{C}}}
\renewcommand{\H}{\ensuremath{\mathbb{H}}}
\newcommand{\bbS}{\ensuremath{\mathbb{S}}}
\newcommand{\M}{\ensuremath{\mathbb{M}}}
\newcommand{\N}{\ensuremath{\mathbb{N}}}

\newcommand{\mf}{\ensuremath{\mathfrak}}
\newcommand{\mc}{\ensuremath{\mathcal}}
\newcommand{\Z}{\ensuremath{\mathbb{Z}}}
\newcommand{\R}{\ensuremath{\mathbb{R}}}
\newcommand{\inv}{\ensuremath{^{-1}}}
\newcommand{\norm}[1]{\left\|#1\right\|}

\renewcommand{\d}[1][t]{\ensuremath{\left.\frac{d}{d#1}\right|_{#1=0}}}
\renewcommand{\Re}{\ensuremath{\operatorname{Re}}}

\DeclareMathOperator{\res}{Res}

\DeclareMathOperator{\interior}{int}
\DeclareMathOperator{\ext}{ext}

\DeclareMathOperator{\isom}{Isom}
\DeclareMathOperator{\dom}{dom}
\title[Spectral Asymptotics for Kinetic Brownian Motion]{Spectral Asymptotics for Kinetic Brownian Motion on Hyperbolic Surfaces}

\author{Martin Kolb, Tobias Weich, and Lasse L. Wolf}
\email{kolb@math.uni-paderborn.de, weich@math.uni-paderborn.de, llwolf@math.uni-paderborn.de}
\begin{document}
\begin{abstract}
 The kinetic Brownian motion on the sphere bundle of a Riemannian manifold~$\M$ is a stochastic process that models a random perturbation  of the geodesic flow. 
 If $\M$ is a orientable compact constant negatively curved surface, we show that in the limit of infinitely large perturbation the $L^2$-spectrum of the infinitesimal generator of a time rescaled version of the process converges to the Laplace spectrum of the base manifold. In addition, we give explicit error estimates for the convergence to equilibrium. 
 The proofs are based on noncommutative harmonic analysis of $SL_2(\R)$.
\end{abstract}

\maketitle

\section{Introduction}
Kinetic Brownian motion is a stochastic process that describes a stochastic perturbation of the geodesic flow and has the property that the perturbation affects only the direction of the velocity but preserves its absolute value. It has been studied in the past years by several authors in pure mathematics \cite{FLJ07, angst, Li16,alexis, BT18} but versions of this diffusion process have been developed independently as surrogate models for certain textile production processes  (see e.g. \cite{GKMW07,GS13, KSW13}). 

Kinetic Brownian motion $(Y_t^\g)_{t\geq 0}$ in the setting of a compact Riemannian manifold $(\M, g)$ can be informally described in the following way: $(Y_t^\g)_{t\geq 0}$ is a stochastic process with continuous paths described by a stochastic perturbation of the geodesic flow on the sphere bundle $S\M =\{\xi\in T\M, \|\xi\|_g=1\}.$
More precisely, if we denote the geodesic flow vector field by $X$ and the (positive) Laplace operator on the fibers of $S\M$ by $\Delta_\bbS$, then the kinetic Brownian motion is generated by the differential operator 
\[
\widetilde P_\g = -X +\frac 12 \g \Delta_\bbS\colon L^2(S\M)\to L^2(S\M).
\]
The connection to the stochastic process $(Y_t^\g)_{t\geq 0}$ is given via
\[
e^{-t\widetilde P_\g}f(x) = \mathbb E_x[f(Y_t^\g)] \quad\text{with}\quad f\in L^2(S\M), x\in S\M.
\]
Observe that the parameter $\g>0$ controls the strength of the stochastic perturbation and it is a natural question to study the behavior of $\widetilde P_\g$ and $Y_t^\g$ in the regimes $\g\to 0$ as well as $\g\to\infty$. By hypoellipticity of $\widetilde P_\g$ one can show that $\widetilde P_\g$ has discrete $L^2$-spectrum. For negatively curved manifolds, Drouot \cite{alexis} has studied the convergence of this discrete spectrum  of $\widetilde P_\g$ in the limit $\g\to 0$ and has shown that it converges to the Pollicott-Ruelle resonances of the geodesic flow. These resonances are a replacement of the spectrum of $X$ since its $L^2$-spectrum is equal to $i\R$ and they can be defined in various generalities of hyperbolic flows as pole of the meromorphically continued resolvent \cite{Liv04, FS11, DZ16a, DG16, DR16,BW17}.  
In the limit of large random noise Li \cite{Li16} and Angst-Bailleul-Tardif \cite{angst} proved  that $\pi(Y_{\g t}^\g)$ converges weakly to the Brownian motion on $ \M$ with speed 2 as $\g\to \infty$ where $\pi\colon S\M\to \M$ is the projection. 
This rescaled kinetic Brownian motion is generated by $P_\g =\g\widetilde P_\g$ whereas the Brownian motion on the base manifold is generated by the Laplace operator $\frac 12 \Delta_\M$. 
Therefore, one may conjecture that the discrete spectrum of $P_\g$ converges to the Laplace spectrum. 
We will give a proof of this fact including explicit error estimates, in the case of constant negative curvature surfaces:
 \begin{theorem}\label{thm:evofPg}
Let $(\M, g)$ be a orientable  compact surface of constant negative curvature scaled to $\kappa = -1$. For every $\eta\in\sigma(\Delta_\M)$ with multiplicity $n$ there is   an analytic function $\lambda_ \eta \colon ]2\sqrt{4\eta +6},\infty[ \to \C$ such that $\lambda_\eta(\gamma)$ is an eigenvalue of $P_\gamma$ with multiplicity at least $n$ and for every $\g > 2\sqrt{4\eta+6}$ the following estimate holds: 
\begin{equation}\label{eq:thm_error_estimate}
|\lambda_\eta (\gamma)- \eta|\leq \frac {8\eta +12}{\g ((4 \eta+6)^{-1/2}-2\g\inv)}. 
\end{equation}
 A fortiori, $\lambda_\eta (\g) \to \eta$ as $\g\to\infty$.
\end{theorem}

Another question to ask is whether the kinetic Brownian motion converges to equilibrium, i.e. 
\[\mathbb E_x[f(Y_{\g t}^\g)] \stackrel{t\to\infty}\longrightarrow \int_ {S\M}f.
\]
Baudoin-Tardif \cite{BT18} showed exponential convergence, i.e.
\[\left \|e^{-t  P_\g} f- \int_{S\M} f \right \|\leq C e^{-C_\g t} \left \|f-\int_{S\M}f\right \|, \quad f\in L^2(S\M).\]
We should point out that the given rate $C_\g$ converges to 0 as $\g\to\infty$ but they conjecture that the optimal rate converges to the spectral gap of $\Delta_\M$ which is the smallest non-zero Laplace eigenvalue $\eta_1$ (see \cite[Section 3.1]{BT18}). 
A direct consequence of Theorem \ref{thm:evofPg} shows that the optimal rate $C_\g$ is less than $\Re \lambda_{\eta_1}(\g)$ for  surfaces of constant negative curvature. Hence $\limsup_{\g\to \infty} C_\g\leq \eta_1$.
For a more explicit study of the convergence towards equilibrium we prove the following spectral expansion:
 \begin{theorem}\label{thm:convergencetoequilibrium}
 Let $(\M, g)$ be a orientable  compact surface of constant negative curvature scaled to $\kappa = -1$.
 For all $ \varepsilon > 0$, $ \g > \max \{4\sqrt{4C{\varepsilon} \inv +6} , 4\sqrt{32}\}$, and $f\in H^2(S\M)$ with  $\|f\|_{H^2(S\M)}\leq C$ (for the precise definition of the used Sobolev norm see Section \ref{sec:hyperbolicsurfaces})
 it holds 
 \[ \bigg \|e^{-tP_\g}f -\sum_{\substack{\eta \in \sigma(\Delta_\M)\\  \eta \leq C{\varepsilon}\inv}}e^{-t \lambda_\eta (\g)}\Pi_{\lambda_\eta(\g)} f \bigg \|_{L^2(S\M)}\leq \varepsilon + \frac 8 {\g^2t}e^{-\g^2t/4} \|f\|_{L^2(S\M)}\]
 where $\lambda_\eta(\g)$ is an eigenvalue of $P_\g$ converging to $ \eta$  as $\g\to \infty$ from Theorem \ref{thm:evofPg} and $\Pi_{\lambda_\eta(\g)}$ is a spectral projector for $P_\g$ of operator norm less than 2.
\end{theorem}
Note that this does not provide an asymptotic expansion for $t\to\infty$ due to arbitrarily small but constant error term $\varepsilon$. However, in contrast to asymptotic expansion in general, all coefficients, including the remainder term are explicitly controllable. As a corollary we get 
an estimate on $\|e^{-t  P_\g} f- \int_{S\M} f  \|$.
\begin{korollar}\label{cor:equilibrium}
  Let $(\M, g)$ be a orientable   compact surface of constant negative curvature scaled to $\kappa = -1$. There is a constant $C_0$ such that for all $C>0,\varepsilon > 0,B\geq 1$ and $f\in H^2(S\M)$ with $\|f\|_{H^2(S\M)}\leq C$ and $\g >\max \{4B(4C{\varepsilon} \inv +6)^{3/2} , 4\sqrt{32}\}$ it holds
  
 \begin{align*}
\norm{e^{-tP_\g}f - \int_{S\M} fd\mu}_{L^2(S\M)} \leq \varepsilon  &+ C_0 C \varepsilon\inv e^{-t(\eta_1-B\inv)}\|f\|_{L^2(S\M)}\\
&+\frac 8 {\g^2t}e^{-\g^2t/4} \|f\|_{L^2(S\M)}
\end{align*} where $\eta_1\coloneqq \min \sigma(\Delta_\M)\setminus\{0\}$.
\end{korollar}

Note that a problem related to the kinetic Brownian motion in $S\M$ is the study of the hypoelliptic Laplacian on $T\M$ introduced by Bismut \cite{Bis05}. Like the kinetic Brownian motion the hypoelliptic Laplacian interpolates between the geodesic flow and the Brownian motion. In \cite[Chapter 17]{BL08} Bismut and Lebeau prove the convergence of the spectrum of the hypoelliptic Laplacian to the spectrum of the Laplacian on $\M$ using semiclassical analysis. It seems plausible that their techniques can also be transferred to the setting of kinetic Brownian motion and might give the spectral convergence without any curvature restriction. The purpose of this article is however not to attack this general setting but show that under the assumption of constant negative curvature, harmonic analysis allows to drastically reduce the analytical difficulties. In fact we are able to reduce the problem to standard perturbation theory. This is also the reason why we are able to obtain the explicit error estimates \eqref{eq:thm_error_estimate}. The approach of applying harmonic analysis to spectral problems related to geodesics flows on manifolds of constant negative curvature (or more generally locally symmetric spaces) has been also pursude in \cite{FF03, DFG15, GHW18, GHW18a, KW17} for the analysis of Pollicott-Ruelle resonances and these results have been a major motivation for the present article.

Let us give a short outline of the proof of Theorem~\ref{thm:evofPg}: By the assumption of constant negative curvature, the manifold $\M$ is up to scaling of the Riemannian metric isometrically isomorphic to a hyperbolic surface  $\G\backslash \H$ where $\G\leq PSL_2(\R)$ is a cocompact torsion-free discrete subgroup and $\H$ the upper half plane.  $\H$ itself can be written as homogeneous space  $PSL_2(\R)/PSO(2)$. 
Under these identifications also the sphere bundle can be written as a homogeneous space $S\M =\G\backslash PSL_2(\R)$ which is obviously a homogeneous space for $PSL_2(\R)$.
Since the manifold is compact we can decompose the corresponding $L^2(S\M)$ into unitary irreducible $PSL_2(\R)$-representations $\mc H_\pi$ and the generator $P_\g$ can be expressed by the right $\mf {sl}_2(\R)$-action.
As a consequence $P_\g$ preserves the decomposition $L^2(S\M)=\oplus \mc H_\pi$ and we can study the restriction $P_\g\colon \mc H_\pi\to\mc H_\pi$ for each occurring representation separately. In each of these irreducible representations the spectral asymptotics of $P_\gamma$ can then be handled by standard perturbation theory of an operator-family of type (A)  in the sense of Kato. 

We want to point out that Theorem~\ref{thm:evofPg} can be obtained without using the representation theory of $SL_2(\R)$. Even more, in an updated article \cite{kww20} we prove Theorem~\ref{thm:evofPg} in the case of constant curvature surfaces, i.e. we extended our result to flat and positively curved surfaces. Here we do not use the representation theory of the corresponding isometry groups of the universal cover. Instead the results are proven using eigenspace decompositions of certain commuting differential operators. 

The article is organized as follows:
We will give a short overview over the kinetic Brownian motion and the connection between constant curvature surfaces and the representation theory of $PSL_2(\R)$  in Section \ref{sec:preliminaries}.
After that we will recall a few results of perturbation theory for unbounded linear operators (Section \ref{sec:pertth}) which are mostly taken from \cite{kato}. 
In the limit $\g\to\infty$ one would like to consider the geodesic vector field as a perturbation of the spherical Laplacian. 
The major difficulty is that $\frac{1}{\g}X$ is not a small perturbation in comparison with $\Delta_\bbS$. 
After the symmetry reduction there is a precise way to consider $X$ as small operator in any irreducible component.
Afterwards we will give a proof of the convergence of the spectra (Theorem \ref{thm:evofPg}). 
In the last part (Section \ref{sec:equilibrium}) we will prove Theorem \ref{thm:convergencetoequilibrium}.

\section{Preliminaries}\label{sec:preliminaries}
\subsection{Kinetic Brownian Motion}\label{sec:kbb}
Let $\M$ be a compact Riemannian manifold of dimension $d\geq 2$ with sphere bundle $S\M=\{(x,v)\in T\M\mid \|v\| =1\}$. We introduce the spherical Laplacian $\Delta_\bbS$ as follows: for every $x \in \M$ the tangent space $T_x\M$ is a Euclidean vector space via the Riemannian metric and $S_x\M=\{v\in T_x\M\mid \|v\|=1\}$ is a submanifold of $T_x\M$. The inner product on $T_x\M$ induces a Riemannian structure on $S_x\M$. Hence, the (positive) Laplace-Beltrami operator $\Delta_\bbS (x)\coloneqq \Delta_{S_x\M}$  of $S_x\M$ defines an operator $C^\infty(S_x\M)\to C^\infty (S_x\M)$. We now obtain the spherical Laplace operator $\Delta_\bbS$ by  
\[\Delta_\bbS: C^\infty(S\M)\to C^\infty(S\M),\quad \Delta_\bbS f (x,v):= (\Delta_\bbS(x)f(x,\cdot))(v).\]

For $(x,v)\in S\M$ and $w \in T _{(x,v)} S\M$ we define $\theta_{(x,v)} (w) =g_x (v, T_{(x,v)} \pi \,w)$ where $\pi \colon S\M\to \M$ is the projection and $g$ is the Riemannian metric on $\M$. Then $\theta$ is a 1-form on $S\M$ and $\nu =\theta \wedge (d\theta)^{d-1}$ defines the Liouville measure on $S\M$ which is invariant under the geodesic flow $\phi_t$. The vector field $X = \d \phi_t^\ast$ is called the geodesic vector field.

Let us consider the operator $P_\g = -\g X + \frac {\g^2}2 \Delta_\bbS$ with domain $\dom(P_\g)=\{u\in L^2(S\M) \mid P_\g u\in L^2(S\M)\}$ for $\g > 0$. Note that the action of $P_\g $ has to be interpreted in the sense of distributions. We first want to collect some  properties of $P_\g$.

\begin{proposition}
\label{prop:kbb}
 $P_\g$ is a hypoelliptic operator with 
 $$\norm{f}_{H^{2/3}}\leq C(\norm f_{L^2}+\norm {P_\gamma f}_{L^2})\quad \text{for}\quad f\in \dom (P_\gamma).$$
  $P_\gamma$ is accretive (i.e. $\Re\langle P_\g f ,f\rangle \geq 0$) and coincides with the closure of $P_\g|_{C^\infty}$.
 Therefore, $P_\g$ has compact resolvent, discrete spectrum with eigenspaces of finite dimension, and the spectrum is contained in the right half plane. $P_\g$ generates a positive strongly continuous contraction semigroup $e^{-tP_\g}$.
\end{proposition}
 \begin{proof}
  See Appendix.
 \end{proof}

\subsection{\texorpdfstring{Representation Theory of $SL_2(\R)$}{Representation Theory of SL(2,R)}}\label{sec:sl2}

\begin{definition}
 The \emph{special linear group} $SL_2(\R)$ is defined by \[SL_2(\R)\coloneqq \left \{\begin{pmatrix}
                                                                                       a&b\\c&d
                                                                                      \end{pmatrix}\in \R^{2\times2} \colon ad-bc=1\right\}.\]
 and the \emph{projective special linear group} by $PSL_2(\R)\coloneqq SL_2(\R) / \{\pm I\}$.
  We abbreviate $PSL_2(\R)$ by $G$.

Both groups are Lie groups with Lie algebra \[\mf g\coloneqq \mf{sl}_2(\R)= \left \{\begin{pmatrix}
                                                                                       a&b\\c&d
                                                                                      \end{pmatrix} \in \R^{2\times2}\colon a+d=0\right\}.\]
\end{definition}

\begin{notation}
We introduce the following elements of $\mf g$ resp. $\mf g \otimes\C$.

\[ \Xi=\frac 12\begin{pmatrix}0&1\\-1&0\end{pmatrix},\quad H=\frac{1}{2}\begin{pmatrix}
      1&0\\0&-1
     \end{pmatrix}, \quad  
     B=\frac{1}{2}\begin{pmatrix}
        0&1\\1&0
       \end{pmatrix}\text{ and }  X_\pm = -H\mp i B.
       \]
       
       The following commutator relations hold:
       $$ [\Xi, H] = -B, \quad [\Xi, B] = H, \quad [H,B] = \Xi, $$ $$ [\Xi,X_\pm] =\pm i X_\pm, \quad [X_+,X_-] = -2i\Xi.$$
       The Casimir element  is given by $$\Omega =4\Xi^2-4H^2-4B^2=4\Xi^2-2(X_+X_-+X_-X_+)\in \mc U(\mf g).$$
       The maximal compact subgroup $K$ of $G$ is $PSO(2)\coloneqq \{\exp (\theta \Xi)\mid \theta \in\R\} / \{\pm I\}$.

\end{notation}

It follows by a simple calculation that $$[\Omega, \Xi] = [\Omega , H] = [\Omega, B] = 0, $$
hence $\Omega \in Z(\mc U(\mf g))$.

Let $(\pi, \mc H_\pi)$  be a irreducible unitary representation of $PSL_2(\R)$. 
Then $\pi(\Omega)$ acts as a scalar $\lambda_\pi$ on $\mc H_\pi$ by Schur's lemma.

Since $PSO(2)$ is compact, $\mc H_\pi$ decomposes as a $PSO(2)$-representation, i.e. we have a orthogonal direct sum \begin{equation}
 \mc H_\pi =\widehat\bigoplus_{k\in \Z} V_k \quad\text{with}\quad \pi(\exp (\theta \Xi)) = e^{ik\theta} \quad \text {on} \quad  V_k.\label{eq:K-types}\end{equation}

One can show that each $V_k$ consists of analytic vectors for $\pi$ and is at most one-dimensional. Let $\phi_k$ denote a normalized element in $V_k$ if $V_k\neq 0$. In particular, $\pi(\Xi) \phi_k = ik \phi_k$ on $V_k$.

The operators $X_\pm$ are raising resp. lowering operators that is $X_\pm \colon V_k\to V_{k\pm 1}$. Indeed, 
$$\Xi X_\pm v = X_\pm \Xi v +[\Xi, X_\pm]v = ik X_\pm v \pm i X_\pm v = i(k\pm 1) X_\pm v, \quad v\in V_k.$$
Moreover, 
$$-4X_\mp X_\pm = \Omega -4\Xi^2 \mp 4i \Xi = \lambda_\pi + 4k^2 \pm 4k = (2k \pm 1)^2 +\lambda_\pi -1.$$
Since $X_\pm^\ast = - X_\mp$, the norm of $X_\pm$ is given by \[\|X_\pm\|_{V_k\to V_{k\pm 1}} =\frac 12 \sqrt{(2k\pm1)^2+\lambda_\pi -1}.\]

The scalar $\lambda_\pi$ classifies  all unitary irreducible representations of $G$.

\begin{theorem}[see {\cite[Ch. 8 Thm. 2.2]{taylor}}]
 Each non-trivial irreducible  unitary representation of $PSL_2(\R)$ is unitarily equivalent to one of the following types:
\begin{itemize}
 \item (Anti-)Holomorphic discrete series: $\pi^\pm_{\pm 2n}$, $n\in\N$, with $\pi^\pm_{\pm 2n}(\Omega)=1-(2n-1)^2$ and $\frac 1i \sigma (\pi^\pm_{\pm 2n}(\Xi))=\pm (n+\N_0)$ 
 \item Principle series: $\pi_{is},$ $s\in\R$, with $\pi_{is}(\Omega)=1+s^2$ and $\frac 1i \sigma (\pi_{is}(\Xi))=\Z$ 
 
 \item Complementary series: $\pi_s$, $s\in (-1,1)\setminus\{0\}$ with $\pi_{s}(\Omega)=1-s^2$ and $\frac 1i \sigma (\pi_{s}(\Xi))=\Z$.
\end{itemize}
There are no unitary equivalences except for  $\pi_{is}\simeq \pi_{-is}$ and  $\pi_{s}\simeq  \pi_{-s}$.
\end{theorem}

In our setting we do not have to distinguish between principle and complementary series representations. Hence, we only distinguish between irreducible unitary representations $\pi$ with $\lambda_\pi <0$ and $\lambda_\pi > 0$ (and the trivial representation). In the former case we have $\mc H_{\pi^\pm_{\pm 2n}} = \bigoplus _ {\pm k\geq n} V_k$ and in the latter case we have $\mc H_\pi =\bigoplus_{k\in \Z} V_k$ with $\dim V_k =1$ for all $\pm k \geq n$ resp. $k\in \Z$.

\subsubsection{Sobolev Regularity for Unitary Representations}\label{sec:sobolev}

Let  $(\pi ,\mc H_\pi)$ be  a unitary representation of a real Lie group $G$ and  $X_1,\ldots , X_n$ be a basis of $\mf g$.
We define the Laplacian $\Delta$ (depending on the Basis) as $$\Delta = -\sum X_i^2.$$ The Laplacian acts as an essentially self-adjoint operator 
on $\mc H_\pi$. The Sobolev space $\mc H_\pi^2$ of order $2$ is the domain of the closure of $I+\Delta$, i.e.
$\mc H_\pi^2=\{u \in \mc H_\pi \mid (I+\Delta) u \in \mc H_\pi\}.$ Here $(I+\Delta) u$ is seen as an element of $(C^\infty (\mc H_\pi))^\ast$ where $C^\infty (\mc H_\pi)$ denotes the set of smooth vectors for $\pi$.  
$ \mc H_\pi^2$ is a Hilbert space with the inner product $\langle u_1, u_2\rangle_2 = \langle (I+\Delta) u_1,(I+\Delta) u_2\rangle$. 

Let $\mc U_k(\mf g_\C)$ be the subspace of $\mc U(\mf g_\C)$ spanned by $Y_{1}\cdots Y_{l}$ with $Y_{i}\in \mf g_\C$ and $l\leq k$.
By \cite[Lemma 6.1]{Nel} we have $$ \forall \,B  \in \mc U_2(\mf g_\C)\,\exists\, C>0 \colon \quad \norm{Bu}\leq C\norm{(I+\Delta)u} \quad \forall\, u\in C^\infty(\mc H_\pi).$$

In particular, $\mc H^2_\pi$ is independent of the choice of basis (in contrast to $\langle\cdot,\cdot\rangle_2$ which depends on the choice of the basis). We will need a slightly more general lemma which is an analogous to the ordinary elliptic regularity estimates in $\R^n$ (see e.g. \cite[Thm. 7.1]{zworski}).

\begin{lemma}
 Let $Q =\Delta + A$ with $A\in \mc U_1(\mf g_\C)$. Then $\mc H^2_\pi=\{u\in\mc H_\pi \mid Qu\in \mc H_\pi\}$ and there is $C>0$ s.t. $\norm{u}_2 \leq C(\norm{Qu}+\norm u)$. 
\end{lemma}

\begin{proof}
Since $\norm{\Delta u } \leq \norm{Qu} + \norm{Au}$ for $u \in C^\infty(\mc H_\pi)$ it remains to show that $\norm{X_i u} \leq C(\norm{Qu}+\norm{u})$. Let therefore $A=\sum_{i=1}^n a_i X_i + b$ with $a_i,b\in \C$.
 Then we have by the Cauchy-Schwarz inequality
 \begin{align*}
  - |a_i\langle X_i u, u\rangle | &\geq - \norm{ X_i u} |a_i| \norm{u} = \frac 12 ((\norm{X_i u}-|a_i|\norm u)^2 -\norm{X_i u}^2-|a_i|^2 \norm{u}^2) \\& \geq -\frac 12 \norm{X_i u}^2 - \frac 12 |a_i|^2\norm{u}^2.
 \end{align*}
 Since $\norm{Qu - u}^2\geq 0$ we infer that
 \begin{align*}
  \norm{Qu}^2 + \norm u ^2 &\geq 2 \Re \langle Qu, u\rangle \\
  &=2\sum  \langle X_i u, X_i u\rangle + 2\sum \Re ( a_i \langle X_i u, u\rangle) + 2\Re b\langle u,u\rangle \\
  &\geq 2\sum \norm{X_i u}^2 - 2\sum |a_i \langle X_i u, u\rangle| - 2|b| \norm{u}^2 \\
  &\geq  \sum \norm{X_i u}^2 - \left (\sum |a_i|^2+2|b| \right ) \norm u ^2.
 \end{align*}
It follows that $\sum \norm{X_i u}^2 \leq \norm{Qu}^2 + \left( 1+ \sum |a_i|^2 + 2|b| \right ) \norm u ^2$. This completes the proof.
\end{proof}

So far we have considered arbitrary unitary representations. Now let $(\pi, \mc H_\pi)$ be an irreducible unitary representation of $G=PSL_2(\R)$. Consider the basis $\Xi, H,B$ of $\mf g$. Then we have $\Delta = -\Xi^2-H^2-B^2 = -2\Xi^2 + \Omega/4$. Note that  $\Omega$ acts as a scalar since $\pi$ is irreducible. Hence, $H^2(\mc H_\pi) = \{u\in \mc H_\pi\mid \Xi^2u\in \mc H_\pi\} =  \{u\in \mc H_\pi\mid (-\Xi^2 + A)u \in \mc H_\pi\}$ for every $A\in \mc U_1(\mf g_\C)$.

\subsection{Hyperbolic Surfaces}\label{sec:hyperbolicsurfaces}

Let $\M$ be a orientable compact  Riemannian manifold of dimension 2 and constant negative curvature $-1$. 
Since $\M$ has finitely many connected components, let us  assume without loss of generality that $\M$ is connected.
By the uniformization theorem %\cite[Cor. 11.13]{lee}
$\M$ is isometrically isomorphic to $\G\backslash \H$ where $\H = \{x+iy\mid y>0\}$ is the upper half plane with the metric $y^{-2} dxdy$ and $\G \subseteq \isom^+(\H)$ is a discrete subgroup of orientation preserving isometries on $\H$ acting freely and properly discontinuously on $\H$. 
 Note that $G=PSL_2(\R)$ acts on $\H$ by the M\"obius transformation. Even more,  $G$ is the  group of orientation preserving isometries  $\isom^+(\H)$ which acts transitively on $\H$. With this action $G/K\simeq\H$ and $ G \simeq S\H $ via $g.(z,v)=(g.z, T_zg\,v)$. 
 
 We infer that $\M$ is a locally symmetric space $\G\backslash G/K$ with sphere bundle $S\M = \G\backslash G$. 
 We have a unitary representation of $G$ on $L^2(\G\backslash G, m)$ (with the Haar measure $m$ on $\G\backslash G$) given by
 \[g.f(x)\coloneqq f(xg),\quad\quad f\in L^2(\G\backslash G),\, x\in  \G\backslash G,\,g\in G\]
 which we call \emph{regular representation}.
 We obtain a Lie algebra representation of $\mf g$ on $C^\infty(\G\backslash G)$ by derivation:
 \[Af(x) = \d f(x\exp (tA)),\qquad f\in C^\infty(\G\backslash G),\, x\in \G\backslash G, \, A\in\mf g.\]
  The geodesic vector field  and the (spherical) Laplacian  can be expressed by elements of $\mc U(\mf g)$.
  \begin{proposition}\label{prop:geometricoperators}
 The geodesic vector field $X$, the spherical Laplace operator $\Delta_\bbS$ on $S\M$ and the Laplace operator on $\M$ are given by \[X = H, \qquad \Delta_\bbS = -\Xi^2\qquad \text{and}\qquad \Delta_\M=-H^2-B^2.\]
 Note that $-H^2-B^2$ is a $K$-invariant element in $\mc U(\mf g)$ so that it defines a right $K$-invariant differential operator on $C^\infty(\G\backslash G)$ that descends to a differential operator on  $C^\infty(\G \backslash G/K)=C^\infty(\M)$.
\end{proposition}
\begin{proof} See Appendix.
\end{proof}

As a consequence, the Haar measure $m$ converts to a $\varphi_t$-invariant smooth measure on $S\M$ under the identification $S\M\simeq \G\backslash G $. Recall that the Liouville measure $\mu$ provides as well a $\varphi_t$-invariant smooth measure. As geodesic flows on compact negatively curved manifolds are known to have a unique smooth invariant measure (up to scaling) the Haar measure can be suitably scaled such that it coincides with the Liouville measure.

 Not only the geometric operators can be expressed by the $\mc U(\mf g)$-action but also the Sobolev spaces $H^\alpha(S\M)$ can be described   in terms of the $\mc U(\mf g)$-action.  More precisely, $-H^2-B^2-\Xi^2$ is an elliptic operator on $L^2(S\M)$ so that we have $$H^\alpha(S\M)= \{u \in L^2(S\M)\mid (I-H^2-B^2-\Xi^2)^{\alpha/2} u \in L^2(S\M)\}.$$
 In particular, $\langle u_1, u_2\rangle_\alpha = \langle (I-H^2-B^2-\Xi^2)^{\alpha} u_1, u_2\rangle$ is a possible choice for an inner product on $H^\alpha(S\M)$ that we will use for our results.
\subsection{Direct Decompositions}\label{sec:decomp}

 Our main tool to investigate the spectrum of the kinetic Brownian motion on $L^2(S\M)$ will be the following theorem.

\begin{theorem}[{see \cite[Ch. 8.6]{taylor}}]\label{thm:decomp}
 The regular representation on $L^2(S\M)$ decomposes discretely into unitary irreducible representation of $G$. For a principle or complementary series representation $\pi$ the multiplicity in $L^2(S\M)$ is given by the multiplicity of the eigenvalue $\frac 14\lambda_\pi$ of the Laplace operator $\Delta_\M$ on $\M$. Moreover, the multiplicity of $\pi^\pm_{\pm n}$ is $(n-1)(g-1)$ for even $n\geq 4$ and $g$ for $n=2$ where $g$ is the genus of $\M$. The trivial representation occurs once in $L^2(S\M)$.
 Hence, 
 \begin{align*}
  L^2(S\M) = \bigoplus_{s\in (0,1)} m(\pi_s) \mc H_{\pi_{s}} \oplus \bigoplus_{s \geq 0}  m (\pi_{is}) \mc H_{\pi_{is}}\oplus \bigoplus _ {n\in\N} m(\pi_{\pm 2n}^\pm) \mc H_{\pi_{\pm 2n}^\pm} \oplus \C
  \end{align*}
   where $m(\pi_s) = \dim\ker (\Delta_\M - \frac 14(1-s^2))$, $m(\pi_{is}) = \dim\ker (\Delta_\M - \frac 14(1+s^2))$, $m(\pi_{\pm 2}^\pm)=g$ and $m(\pi_{\pm 2n}^\pm)=(2n-1)(g-1)$.  
The Sobolev space decomposes into
\begin{align*}
  H^2(S\M) = \bigoplus_{s\in (0,1)} m(\pi_s) \mc H^2_{\pi_{s}} \oplus \bigoplus_{s \geq 0}  m (\pi_{is}) \mc H^2_{\pi_{is}}\oplus \bigoplus _ {n\in\N} m(\pi_{\pm 2n}^\pm) \mc H^2_{\pi_{\pm 2n}^\pm} \oplus \C.
  \end{align*}
\end{theorem}

\subsection{Perturbation Theory}\label{sec:pertth}

We want to collect some basic results from perturbation theory for linear operators that can be found in \cite{kato}.
First, we introduce families of operators we want to deal with.

\begin{definition}[{see \cite[Ch. VII \S2.1]{kato}}]
 A family $T(x)$ of closed operators on a Banach space $X$ where $x$ is an element in a domain $D\subseteq \C$ is called \emph{holomorphic of type (A)} if the domain of $T(x)$ is independent of $x$ and $T(x)u$ is holomorphic for every $u\in\operatorname{dom}(T(x))$.  
\end{definition}

Without loss of generality let us assume that 0 is contained in the domain $D$. We call $T=T(0)$ the unperturbed operator and $A(x)=T(x)-T$ the perturbation. Furthermore, let $R(\zeta,x)= (T(x)-\zeta)\inv$ be the resolvent of $T(x)$ and $R(\zeta)=R(\zeta,0)$. If $\zeta\notin \sigma(T)$ and $1+A(x)R(\zeta)$ is invertible then $\zeta \notin \sigma(T(x))$ and the following identity holds:
\begin{align}\label{eq:resolventformula2}R(\zeta,x)=R(\zeta)(1+A(x)R(\zeta))\inv.\end{align}

Let us assume that $\sigma(T)$ splits into two parts by a closed simple $C^1$-curve $\G$. Then there is $r>0$ such that $R(\zeta,x)$ exists for $\zeta\in \G$ and $|x|<r$. %\cite[Ch. VII Thm. 1.3]{kato}}
If the perturbation is linear (i.e. $T(x)=T+xA$) then a possible choice for $r$ is given by $\min_{\zeta\in\Gamma}\|AR(\zeta)\|^{-1}$.
In particular, we obtain that $\Gamma\subseteq \C\setminus\sigma(T(x))$ for $|x|<r$, i.e. the spectrum of $T(x)$  still splits into two parts by $\Gamma$. Let us define $\sigma_{\interior}(x)$ as the part of $\sigma(T(x))$ lying inside $\Gamma$ and $\sigma_{\ext}(x)=\sigma(T(x))\setminus\sigma_{\interior}(x)$. 
The decomposition of the spectrum gives a $T(x)$-invariant decomposition of the space $X= M_{\interior}(x)\oplus M_{\ext}(x)$ where $M_{\interior}(x)=P(x)X$ and $M_{\ext}(x)=\ker P(x)$ with the bounded-holomorphic projection \[P(x)=-\frac{1}{2\pi i}\int_\Gamma R(\zeta,x)d\zeta.\] 
Furthermore, $\sigma(T(x)|_{M_{\interior}(x)})=\sigma_{\interior}(x)$ and $\sigma(T(x)|_{M_{\ext}(x)})=\sigma_{\ext}(x)$. %\cite[Ch. VII Thm. 1.7]{kato}
To get rid of the dependence of $x$ in the space $M_{\interior}(x)$ we will use the following proposition.

\begin{proposition}[see {\cite[Ch. II \S4.2]{kato}}]
Let $P(x)$ be a bounded-holomorphic family of projections on a Banach space $X$ defined in a neighbourhood of 0.
 Then there is a bounded-holomorphic family of operators $U(x)\colon X\to X$ such that $U(x)$ is an isomorphism for every $x$ and $U(x)P(0)=P(x)U(x)$. In particular, $U(x) P(0) X = P(x)X$ and $U(x)\ker P(0) = \ker P(x)$.
\end{proposition}

Denoting $U(x)^{-1} T(x) U(x)$ as $\widetilde{T}(x)$ we observe \[\sigma(\widetilde T(x)|_{M_{\interior}(0)})= \sigma(\widetilde T(x)) \cap \operatorname{int}(\G) = \sigma(T(x)) \cap \operatorname{int}(\G)\] since $U(x)$ is an isomorphism. Here we denote the interior of $\G$ by $\operatorname{int}( \G)$.

Let us from now on suppose that $\Gamma$    is a circle with radius $\rho$ centered at an eigenvalue $\mu$ of $T$ with finite multiplicity  and encloses  no other eigenvalues of  $T$. Then $\sigma_{\interior}(0)=\{\mu\}$  and $M_{\interior}(0)$ is finite dimensional.
Hence, $\widetilde T(x)|_{M_{\interior}(0)}$ is a holomorphic family of  operators on a finite dimensional vector space. It follows that the eigenvalues of $T(x)$ are continuous as a function in $x$.
In addition to the previous assumptions, let us suppose that the eigenvalue $\mu$ is simple. 
Then $M_{\interior}(0)$ is one-dimensional and $\widetilde T(x)|_{M_{\interior}(0)}$ is a scalar operator.
We obtain that there is a holomorphic function $\mu\colon B_r\to\C$ (with $r=\min_{\zeta\in\Gamma}\|AR(\zeta)\|^{-1}$ as above) such that $\mu(x)$ is an eigenvalue of $T(x)$, $\mu(x)$ is inside $\Gamma$ and $\mu(x)$ is the only part of $\sigma(T(x))$ inside $\Gamma$ since $\sigma_{\interior}(x)=\sigma(\widetilde T(x)|_{M_{\interior}(0)})$. 
As a consequence, \[|\mu (x) -\mu|<\rho \qquad \forall\ |x|<r.\] By  Cauchy's inequality we infer $|\mu^{(n)}|\leq \rho r^{-n}$ for the Taylor series $\mu(x)=\sum  x^n \mu^{(n)}$. Hence, 
\begin{align}\left |\mu(x)-\sum_{n=0}^N x^n \mu^{(n)}\right|\leq \rho\cdot \frac{|x|^{N+1}}{r^{N}(r-|x|)}\quad \forall\ |x|<r.\label{eq:error_estimate}
\end{align}

We now want to calculate the Taylor coefficients of $\mu(x)$ in order to get an approximation of $\mu(x)$  in the case where $X=\mc H$ is a Hilbert space and $T(x)$ is a holomorphic family of type (A) with symmetric $T$ but not necessarily symmetric $T(x)$ for $x\neq 0$.
To this end let $\varphi(x)$ be a normalized holomorphic family of  eigenvectors (obtained from $P(x)$).
Consider the Taylor series $\mu(x)=\sum  x^n \mu^{(n)}$,  $\varphi(x)=\sum  x^n \varphi^{(n)}$ and  $T(x)u=\sum x^n T^{(n)} u $ for every $u \in \dom(T)$ which converges on a disc of positive radius independent of $u$. This is due to the fact that Taylor series of holomorphic functions converge on every disc that is contained in the domain.

We compare the Taylor coefficients in \begin{align*}
(T(x)-\mu(x))\varphi(x)=0\qquad \text{and} \qquad\langle(T(x)-\mu(x))\varphi(x),\varphi(x)\rangle=0
\end{align*}
and obtain \begin{equation*}
(T-\mu^{(0)})\varphi^{(l)}=-\sum_{n=1}^l(T^{(n)}-\mu^{(n)})\varphi^{(l-n)}\end{equation*} and
\begin{align*}
\mu^{(k)}= &\langle T^{(k)}\varphi^{(0)},\varphi^{(0)}\rangle+\sum_{n=1}^{k-1}\langle(T^{(n)}-\mu^{(n)})\varphi^{(k-n)}, \varphi^{(0)}\rangle.
\end{align*}
A fortiori,
\begin{align}
\mu^{(1)}&= \langle T^{(1)}\varphi^{(0)},\varphi^{(0)}\rangle \label{eq:firstderivative} \\ 
\mu^{(2)}&= \langle T^{(2)}\varphi^{(0)},\varphi^{(0)}\rangle+\langle(T^{(1)}-\mu^{(1)})\varphi^{(1)}, \varphi^{(0)}\rangle, \label{eq:secondderivative} \end{align}
where $\varphi^{(1)}$ fulfils \begin{equation}(T-\mu^{(0)})\varphi^{(1)}=-(T^{(1)}-\mu^{(1)})\varphi^{(0)}.\label{eq:derivativevector}
 \end{equation}
  Although $\varphi^{(1)}$ is not uniquely determined by this equation, $\mu^{(2)}$ can be calculated in our setting.
Here $\varphi^{(1)}=v+c\varphi^{(0)}$ with unique $v\in \ker(T^{(0)}-\mu^{(0)})^\perp$ as $T^{(0)}$ is symmetric. 
We infer that 
\begin{align*}
 \langle(T^{(1)}-\mu^{(1)})\varphi^{(1)}, \varphi^{(0)}\rangle&=\langle(T^{(1)}-\mu^{(1)})v, \varphi^{(0)}\rangle-c\mu^{(1)}+c\langle T^{(1)}\varphi^{(0)}, \varphi^{(0)}\rangle \\&= \langle(T^{(1)}-\mu^{(1)})v, \varphi^{(0)}\rangle.
\end{align*}
 Therefore, $\mu^{(2)}$ depends only on $v$ and not on $c$.

\section{Perturbation Theory of the Kinetic Brownian Motion}
\label{sec:der}

 We want to establish the limit $\gamma\to \infty$ of the spectrum of $P_\g$. 
 To do so we write $P_\g= \frac{\g^2}2 (\Delta_\bbS-2\g^{-1} X )=\frac{\g^2}2 T(-2\g^{-1})$ where  $T(x)=\Delta_\bbS+xX$ and we want to use the methods established in Chapter \ref{sec:pertth}. 
 
In order to have finite dimensional eigenspaces and holomorphic families of type (A) we will use the orthogonal decomposition of $L^2(S\M)$ derived in Theorem \ref{thm:decomp}:
\[L^2(S\M)\simeq L^2(\G\backslash G)=\bigoplus\nolimits_{\pi\in\widehat G} m(\pi) \mc H_\pi.\]
Here, $T(x)$ is given by $-\Xi^2 + x H$ by Proposition \ref{prop:geometricoperators}.  We denote the restriction of $T(x)$ to $\mc H_\pi$ by $T_\pi(x)$ and its resolvent by $R_\pi(\zeta,x)$ and $R _\pi(\zeta)=R_\pi(\zeta,0)$.

\begin{remark}
It follows from Section~\ref{sec:sobolev} that $\dom(T_\pi(x))=\{u\in \mc H_\pi \mid T_\pi(x) u \in \mc H_\pi\}=\mc H^2_\pi$. Furthermore,  $T_\pi(x)$ is closed as a restriction of a closed operator. 
  We conclude that $T_\pi(x)$ is a holomorphic family of type (A) on the complex plane with domain $ \mc H^2_\pi$.
\end{remark}

\begin{remark}
 One can realize the principle series representation on $\mc H_{\pi_{is}} = L^2(S^1)$ (see \cite[Ch. 4.3]{taylor}). Here $-\Xi^2$ is taken to $\Delta_{S^1}$ such that $\mc H_{\pi_{is}}^2=H^2(S^1)$. The remark from above  then follows from the elliptic estimate $$\|u\|_{H^2(S^1)}\leq C(\|u\|_{L^2(S^1)}+\| (\Delta_{S^1}+ a(\vartheta) \partial _\vartheta +b(\vartheta) )u\|_{L^2(S^1)})$$ noting that $H$ is a first order differential operator. 
\end{remark}

We use the structure of the $G$-representations to obtain a more precise version of elliptic regularity.

\begin{lemma}\label{la:Tbdd}
 $H$ is $\Xi^2$-bounded on $\mc H_\pi$, more precisely $$\|Hu\|^2\leq \frac{|\lambda_\pi|}{4}\|u\|^2 + \frac 3 2\|\Xi^2 u\|^2 \quad\text{with}\quad u\in\mc H_\pi^2.$$
\end{lemma}
\begin{proof}
 Let us express $u\in \mc H_\pi^2$ in its Fourier expansion according to $K$-types (see \eqref{eq:K-types}), i.e. $u=\sum_{n\in\Z} a_n \phi_n\in\mc H_\pi^2$. Since $H=-\frac 12 (X_++X_-)$ with the raising/lowering operators $X_\pm \colon V_n\to V_{n\pm 1}$ we can compute
 \begin{align*}
  \|Hu\|^2=& \langle -H^2 u,u\rangle\\
    =&-\sum_n a_n \overline{a_n}\langle H^2\phi_n,\phi_n\rangle + a_n \overline{a_{n+2}}\langle H^2 \phi_n,\phi_{n+2}\rangle+ a_n \overline{a_{n-2}} \langle H^2\phi_n,\phi_{n-2}\rangle\\
    =&-\frac 14 \sum_n |a_n|^2 \langle (X_+X_-+X_-X_+)\phi_n,\phi_n\rangle + a_n \overline{a_{n+2}}\langle X_+^2 \phi_n,\phi_{n+2}\rangle+\\
    \phantom=&\phantom{-\frac 14 \sum}+a_n \overline{a_{n-2}} \langle X_-^2\phi_n,\phi_{n-2}\rangle.
 \end{align*}

Since $\Omega =4\Xi^2-2(X_+X_-+X_-X_+)$ and $\Xi=in$ on $V_n$ we infer that \[\langle (X_+X_-+X_-X_+)\phi_n,\phi_n\rangle=-2n^2-\frac 12 \lambda_\pi.\]
Moreover, $\|X_\pm\|_{V_n\to V_{n\pm 1} }= \frac 12 ((2n\pm1)^2+\lambda_\pi -1)^{1/2}$ by Section \ref{sec:sl2}.
% and $\|X_-\|_{V_n\to V_{n-1}} = \frac 12 ((2n-1)^2+\lambda_\pi -1)^{1/2}$
Hence, \[|\langle X_\pm^2\phi_n,\phi_{n\pm 2}\rangle| = \frac 14 ((2n\pm1)^2+\lambda_\pi -1)^{1/2} ((2n\pm3)^2+\lambda_\pi -1)^{1/2}.\]
With the Cauchy-Schwarz-inequality we obtain
\begin{align*}
  \big|\sum_n a_n &\overline{a_{n\pm2}}\langle X_\pm^2 \phi_n,\phi_{n\pm2}\rangle\big|^2\\
   &\leq \frac 1{16} \sum_n |a_n|^2|(2n\pm1)^2+\lambda_\pi -1|\sum_n |a_{n\pm2}|^2|(2n\pm3)^2+\lambda_\pi -1|\\
&=  \frac 1{16} \sum_n |a_n|^2|(2n\pm1)^2+\lambda_\pi -1|\sum_n |a_{n}|^2|(2n\mp1)^2+\lambda_\pi -1|\\
&\leq\frac1{16}\left(\sum_n|a_n|^2(|\lambda_\pi|+4n^2+4|n|)\right)^2.\\
\end{align*}
We conclude \begin{align*}
\|Hu\|^2&\leq\frac 14 \sum_n |a_n|^2 (2n^2+\frac 12 |\lambda_\pi|+ 2\cdot\frac 14 (|\lambda_\pi| + 4n^2 +4|n|))\\
&\leq \frac 14 |\lambda_\pi| \|u\|^2+\frac 32 \|\Xi^2 u\|^2.\qedhere
\end{align*}

\end{proof}

The eigenspaces of the unperturbed operator $-\Xi^2$ are $V_0$ and $V_k\oplus V_{-k}$ which are finite dimensional. 
As we have seen in Section \ref{sec:pertth} the eigenvalues of a holomorphic family of type (A) are continuous as a function of $x$ in this case.   We deduce that for the eigenvalues $\mu(x)$ of $T_\pi(x)$ that arise from non-zero eigenvalues $\mu = \mu(0)$ of $-\Xi^2$ the limit $\gamma\to \infty$ of $\frac{\gamma^2}2\mu(2 \gamma\inv)$, which is an eigenvalue of $P_\gamma$, is $\infty$.
Therefore, we  do not care about non-zero eigenvalues at first.
Since $\Xi$ has non-zero spectrum in the discrete series representation  we start with a principle or complementary series representation $(\pi,\mc H_\pi)$. %Here $V_k\neq 0$ for every $k$.

 Here the eigenspace  for the eigenvalue 0 of $T_\pi(0)$ is  $\langle\phi_0\rangle$ which is one-dimensional. % 
This means  that there is an analytic eigenvalue $\mu(x) = \sum x^n \mu^{(n)}$ of $T_\pi(x)$ and its eigenvector $\varphi(x)=\sum x^n \varphi^{(n)}$ is analytic on some $B_r(0)$ which will be determined later on.
 Note that $\mu^{(0)} =0 $, $\varphi^{(0)}=\phi_0$, $T=T^{(0)}=-\Xi^2$ and $T^{(1)}=H$ in this case.
 We can use Equation \eqref{eq:firstderivative} from Section \ref{sec:pertth}:  
\begin{align*}
\mu^{(1)}=\langle T^{(1)} \varphi^{(0)}, \varphi^{(0)}\rangle=\langle H \phi_0,\phi_0\rangle=\frac12\langle - (X_++X_-)\phi_0,\phi_0\rangle.
\end{align*}
 Due to the fact that $X_\pm$ are raising respectively lowering operators, i.e. $X_\pm V_k\subseteq V_{k\pm1}$, we conclude that $\mu'(0)=\mu^{(1)}=0$. 

We now want to find the second derivative $\mu''(0)= 2 \mu^{(2)}$ of $\mu$. 
According to Section \ref{sec:pertth} we first have to calculate $\varphi^{(1)}$ via $-\Xi^2\varphi^{(1)}=-H\phi_0$ (see Equation~ \eqref{eq:derivativevector}). 
Notice that $-H\phi_0\in V_{-1}\oplus V_1 = \{u\mid -\Xi^2u =u\}$. Furthermore $\ker(-\Xi^2)=V_0 = \langle \phi_0\rangle$, and consequently  $\varphi^{(1)}=-H\phi_0+c\phi_0$ for some $c\in\C$. Let us recall that $\mu''(0)$ is independent of $c$. 
Consequently by Equation \eqref{eq:secondderivative}, \begin{align*}
\mu''(0) = 2 \mu^{(2)}=& 2\langle H(-H\phi_0), \phi_0\rangle=-\frac{1}{2}\langle (X_++X_-)^2\phi_0,\phi_0\rangle\\ =& -\frac12 \langle ( X_+^2+ X_+X_-+ X_-X_++ X_-^2)\phi_0,\phi_0\rangle.
\intertext{Again, $X_\pm$ are raising/lowering operators. Therefore, }
\mu''(0)&=-\frac12\langle (X_+X_- +X_-X_+)\phi_0,\phi_0\rangle \\
&=\frac 14 \langle \Omega \phi_0,\phi_0\rangle = \frac {\lambda_\pi}{4}
\end{align*}
as the Casimir operator $\Omega$ equals $4\Xi^2-2(X_+X_-+X_-X_+)$ and $\Xi \phi_0=0$.

Summarizing, we arrived at the following situation.

\begin{proposition}\label{thm:evofTpi}
 For a principle or complementary series representation $\pi$ there is  $r_\pi>0$ and an analytic function $\mu\colon B_{r_\pi}(0) \to \C$ such that $\mu(x)$ is an eigenvalue of $T_\pi(x)$ with multiplicity 1 and $\mu (x)=x^2\frac 12 \frac{\lambda_\pi}4 + \mc O (x^{3})$. A fortiori, $x^{-2}\mu (x) \to \frac 12 \frac{\lambda_\pi}{4}$ as $x\to 0$.
\end{proposition}

We want to determine error estimates for the eigenvalues  and an lower bound for $r_\pi$ used above.

Let $\G$ be the circle with radius $\frac 12$ centered at 0. 
Hence the spectrum of $T_\pi(0)$  for a principle or complementary series representation is separated by $\G$ where the only eigenvalue inside $\G$ is 0. 
As we have seen in   Section \ref{sec:pertth} a choice for $r_\pi$ is $r_\pi=\min_{\zeta\in\Gamma} \|HR_\pi(\zeta)\|^{-1}$. 

\begin{lemma}\label{la:normhr}
Let $\sigma$ be the spectrum of $-i\Xi$ on $\mc H_\pi$, i.e. $\sigma = \{k\mid k\in \Z\}$ if $\pi = \pi_{is}$ or $\pi = \pi_s$ or $\sigma = \{k\mid\pm k\geq n\}$ if $\pi = \pi_{\pm 2n}^\pm$. For $\zeta \in \C\setminus \sigma^2$ we have
$$ \|R_\pi(\zeta)\| = \sup_{k\in \sigma} |k^2-\zeta|^{-1}.$$
Addionally we can estimate:
\[\|HR_\pi(\zeta)\|^2\leq \left(\frac{|\lambda_\pi|}{4}+3|\zeta|^2\right) \|R_\pi(\zeta)\|^2 +3.\]
\end{lemma}
\begin{proof}
 Let us first evaluate the norm of $R_\pi(\zeta)$. On the one hand $R_\pi(\zeta)\phi_k=(k^2-\zeta)\inv\phi_k$ and we infer $\|R_\pi(\zeta)\|\geq |k^2-\zeta|\inv$ for all $k\in \sigma$. On the other hand, \[\left\|R_\pi(\zeta)u\right\|=\left\|\sum_{k\in \sigma} a_k (k^2-\zeta)\inv \phi_k\right\|=\sqrt{\,\sum_{k\in \sigma} |a_k|^2 |k^2-\zeta|^{-2}}\leq \sup_{k\in \sigma} |k^2-\zeta|\inv \|u\|\]
 for $u=\sum_k a_k\phi_k$. Thus $\|R_\pi(\zeta)\| = \sup_{k\in \sigma} |k^2-\zeta|\inv$.
 
 Using Lemma \ref{la:Tbdd} it follows that 
\begin{align*}
\|HR_\pi(\zeta)\|^2&\leq \frac {|\lambda_\pi|} 4 \|R_\pi(\zeta)\|^2+ \frac 32 \left\|-\Xi^2\left(-\Xi^2 -\zeta\right)\inv \right\|^2\\
&\leq \frac {|\lambda_\pi|} 4 \|R_\pi(\zeta)\|^2+ \frac 32 \left\| 1+\zeta R_\pi(\zeta)\right\|^2\\
&\leq \left(\frac{|\lambda_\pi|}{4}+3|\zeta|^2\right)\|R_\pi(\zeta)\|^2 +3
\end{align*}
where we used $(x+y)^2\leq 2(x^2+y^2)$ in the last step.
\end{proof}

\begin{korollar}
\begin{enumerate}[(i)]
 \item 
 Let $\pi$ be a principle or complementary series representation. Then $R_\pi(\zeta,x)$ exists for all $|\zeta|\geq \frac 12$, $\Re \zeta\leq \frac 12$ and $|x|< (\lambda_\pi+6)^{-1/2}$ and we have $$\| R_\pi(\zeta,x)\|\leq |\zeta|\inv \left(1-|x|\sqrt{\lambda_\pi+6}\right)\inv.$$ 
 \item  
 Let  $\pi$ be a discrete series representation $\pi_{\pm 2n}^\pm$. Then  $R_\pi(\zeta,x)$ exists for all $\Re \zeta\leq \frac 12$ and $|x|<1/\sqrt{32}$ and we have $$\| R_\pi(\zeta,x)\|\leq |\zeta-n^2|\inv \left(1-|x|\sqrt{32}\right)\inv.$$ 
\end{enumerate}
\label{cor:resofT}
\end{korollar}
\begin{proof}
 Let $\Re \zeta\leq \frac 12$ and $|\zeta|\geq \frac 12$. Then $\|R_\pi(\zeta)\| = \sup_{k\in \Z} |k^2 -\zeta|\inv = |\zeta|\inv$ in the first case. 
 A simple consequence of Lemma \ref{la:normhr} is  \[\|HR_\pi(\zeta)\|^2\leq \frac{\lambda_\pi}{4|\zeta|^2} +6\leq \lambda_\pi +6.\]
 Combining this with Equation \eqref{eq:resolventformula2} we infer that $\zeta\not\in \sigma(T_\pi(x))$ for every $x$ with $|x|<\sqrt{\lambda_\pi+6}$.
 The stated estimate is a consequence of Equation \eqref{eq:resolventformula2}, too. 
 
 In the  case of $\pi=\pi_{\pm 2n}^\pm$, we have $\|R_\pi(\zeta)\| = \sup_{k \geq n} |k^2-\zeta|^{-1} = |n^2-\zeta|\inv$ if $\Re \zeta \leq \frac 12$. 
 Consequently by Lemma \ref{la:normhr},  
 \begin{align*}
  \|HR_\pi(\zeta)\|^2&\leq\left(\frac{|\lambda_\pi|}{4}+3|\zeta|^2\right) \|R_\pi(\zeta)\|^2 +3\\
  &= \frac{(2n-1)^2-1}{4|n^2-\zeta|^2}+3\frac{|\zeta|^2}{|n^2-\zeta|^2} + 3\\
  &\leq\frac{n^2-n}{(n^2-1/2)^2}+3\frac{|\zeta|^2}{|1-\zeta|^2} + 3\\
  &\leq \frac{1}{n^2-1/2}+3\left(1+ \frac{1}{|1-\zeta|}\right)^2+3\\
  &\leq 2+3\cdot 9 +3 = 32.
 \end{align*}
Using again Equation \eqref{eq:resolventformula2}  finishes the proof.
\end{proof}

Now we can prove the following theorem on the spectrum of $T_\pi(x)$.

\begin{theorem}
\begin{enumerate}[(i)]
 \item  Let $\pi$ be a principle or complementary series representation and $r_\pi=(\lambda_\pi +6)^{-1/2}$. Then, there is a holomorphic function $\mu\colon B_{r_\pi}(0)\to \C$ such that $\mu(x)$ is an eigenvalue of $T_\pi(x)$ with multiplicity 1, $|\mu(x)|\leq \frac 12$ and $\sigma(T_\pi(x))\cap \{\zeta\mid \Re \zeta\leq \frac 12\} = \{\mu(x)\}$ for all $x\in B_{r_\pi}(0)$. 
 Furthermore, \[\left|\mu(x)-\frac 12 \frac { \lambda_\pi}4 x^2\right|\leq \frac 12 \frac {|x|^3}{r_\pi^2(r_\pi-|x|)}\qquad \forall \ |x|< r_\pi.\]
 \item Let $\pi$ be a discrete series representation. Then $\Re \sigma (T_\pi(x))> \frac 12$ for all $x$ with $|x|<1/\sqrt{32}$.
\end{enumerate}\label{thm:specT}
\end{theorem}
\begin{proof}
 We have seen before that $\mu(x)$ is the only eigenvalue with absolute value smaller than $\frac 12$ if $|x|< \min_{|\zeta|=1/2} \|HR_\pi(\zeta)\|^{-1}$. Since $\|HR_\pi(\zeta)\|\leq \frac 1{r_\pi}$ by Corollary \ref{cor:resofT} this is the case if $|x|<r_\pi$. 
 In Proposition \ref{thm:evofTpi} we calculated $\mu''(0)=\frac{\lambda_\pi}{4}$ and with Equation \eqref{eq:error_estimate} we obtain the error estimate.
 The  statement about the discrete series that remains to be proven follows directly from Corollary \ref{cor:resofT}.
\end{proof}

\begin{remark}\label{bem:nonuniform}
 Unfortunately, the radius $r_\pi$ depends on $\lambda_\pi$ which is given by $1+s^2$ for $\pi = \pi_{is}$. As $\mc H_{\pi_{is}}$ are contained in $L^2(\G\backslash G)$ for arbitrary large $s$  we do not obtain a uniform bound on $r$. 
 Since \[\sup_{|\zeta|=1/2} \|HR_\pi(\zeta)\| \geq \|HR(1/2)\|\geq \|HR(1/2)\phi_0\|=2\|H\phi_0\| =2\sqrt{\frac 12 \frac{\lambda_\pi}4}\]
 we can not get rid of the dependence on $\lambda_\pi$. 
\end{remark}

Reformulated in terms of $x=-2\g\inv$ we obtain Theorem \ref{thm:evofPg} for the generator of the kinetic Brownian motion on $S\M$.

\begin{proof}[Proof of Theorem \ref{thm:evofPg}]
 As we have seen in Theorem \ref{thm:decomp} $L^2(\G \backslash G)$ decomposes discretely in unitary irreducible representations and the multiplicity of a principle or complementary series representation $\pi$ in $L^2(\G\backslash G)$ is given by the multiplicity of the eigenvalue $\frac {\lambda_\pi} 4$ of $\Delta_\M$.
 Thus, if $\eta$ is a $\Delta_\M$-eigenvalue of multiplicity $n$ then there is a principle or complementary series representation $(\pi,\mc H_\pi)$ such that $\eta=\frac{\lambda_\pi}{4}$ and that occurs $n$ times in $L^2(\G\backslash G)$. 
 For this representation Theorem \ref{thm:specT}  states that there is  $\mu\colon B_{r_\pi}(0)\to \C$ for $r_\pi=(\lambda_\pi+6)^{-1/2}=(4\eta +6)^{-1/2}$ such that $\mu(x)$ is an eigenvalue of $T_\pi(x)$. 
 Since $P_\g=\frac{\g ^2}2T(-2\g\inv)$ and $T_\pi$ is the restriction of $T$ to $\mc H_\pi$ we obtain that $\frac{\g^2}2\mu(-2\g\inv)$ is an eigenvalue with multiplicity $n$ of $P_\g$ if $\g>2\sqrt{4\eta+6}$. The given estimate follows from Theorem \ref{thm:specT} as well. 
\end{proof}

\section{Convergence to Equilibrium}\label{sec:equilibrium}
In this chapter we want to analyse the convergence  of the kinetic Brownian motion to equilibrium. 
As it has been mentioned above this convergence is described by the propagator $e^{-tP_\g}$. 
In general, the resolvent $(A + \zeta)\inv$ of a generator $A$ of a contraction semigroup on a Banach space $X$ is the Laplace transform of $e^{-tA}$ by the Hille-Yosida theorem (e.g. \cite[Thm. X.47a]{reedsimon}). Hence, we  can obtain $e^{-tA}$ by the inverse Laplace transform of $(A+\zeta)\inv$.
More precisely we have the following proposition.

\begin{proposition}[{e.g. \cite[Ch. III Cor. 5.15]{engelnagel}}]\label{prop:inversion}
 If $A$ generates the strongly continuous contraction semigroup $e^{-tA}$ on a Banach space $X$ then we have for all $u\in \dom(A)$ and $w<0$:
 \[e^{-tA}u = \frac 1 {2\pi i} \lim_{n\to \infty} \int _{w-in}^{w+in} e^{-\zeta t} R(\zeta) u\,d\zeta.\]
 %The convergence is uniform in $t$ in compact intervals of $]0,\infty[$.
\end{proposition}

Unfortunately, the integral does not converge absolutely. We will solve this issue by using integration by parts and the explicit estimates obtained by Corollary~\ref{cor:resofT}.

\begin{proposition}\label{prop:semigrouprestriction}
 $T_\pi(x)$ generates a contraction semigroup $e^{-tT_\pi(x)}$ for real $x$, $\mc H_\pi$ is $e^{-tP_\g}$-invariant, and we have $$\left. e^{-tP_\g}\right|_{\mc H_\pi} = e^{-(t\frac{\g^2}2)T_\pi(2\g\inv)}.$$
\end{proposition}
\begin{proof}
 Since $T_\pi(x)$ is the restriction of a multiple of $P_{-2x\inv}$ it generates a contraction for real $x$ as well. The last statements follow from Proposition~\ref{prop:inversion} with the observation that $\dom(T_\pi(x)) =\mc H _\pi^2$ is dense in $\mc H_\pi$.
\end{proof}

We are now going to analyse the decay rate of $e^{-tP_\g}$ restricted to a fixed unitary representation. 

\begin{theorem}\label{thm:contractionpi}
 Let $\pi$ be a complementary or principle series representation, $\mu(x)$ the eigenvalue of $T_\pi(x)$ from  Theorem \ref{thm:specT}, $r_\pi = (\lambda_\pi +6)^{-1/2}$   and $$P(x)=-\frac 1{2\pi i} \int_{|\zeta|=1/2} R(\zeta,x)\, d\zeta, \qquad |x|<r_\pi,$$ the projection onto the eigenspace corresponding to $\mu(x)$.
 Then we have 
 \[e^{-tT_\pi(x)} u = e^{-\mu(x)t} P(x)u + \frac 1t \frac{1}{2\pi i} \int_{1/2-i\infty}^{1/2+i\infty}e^{-\zeta t}R_\pi(\zeta,x)^2u\,d\zeta\]
 for all $u\in \mc H_\pi^2$, $x\in \R$ with $|x|<r_\pi$. Furthermore,
 \[\|e^{-tT_\pi(x)} u - e^{-\mu(x)t} P(x)u \| \leq \frac 4t e^{-t/2}\|u\|\]
 if $|x|\leq r_\pi/2$.
 
 If $\pi$ is a discrete series representation $\pi_{\pm 2n}^\pm$ we have 
 \[e^{-tT_\pi(x)} u = \frac 1t \frac{1}{2\pi i} \int_{1/2-i\infty}^{1/2+i\infty}e^{-\zeta t}R_\pi(\zeta,x)^2u\,d\zeta\] and  \[\|e^{-tT_\pi(x)} u  \| \leq \frac 2{t(n^2-1/2)} e^{-t/2}\|u\|\leq \frac 4t e^{-t/2}\|u\|\quad \text{if} \quad |x|\leq \frac 1{2\sqrt{32}}.\]
\end{theorem}
\begin{proof}
 From Proposition \ref{prop:inversion} we obtain that 
 \[e^{-tT_\pi(x)}u  = \frac 1 {2\pi i} \lim_{n\to \infty} \int _{w-in}^{w+in} e^{-\zeta t} R_\pi(\zeta,x) u\,d\zeta\]
 if $w<0$ and $u\in \dom(T_\pi(x))=\mc H_\pi^2$. 
 Since $|x|<r_\pi$ we infer with Theorem \ref{thm:specT} that $\sigma(T_\pi(x)) \cap \{\Re \zeta \leq 1/2\} = \{\mu(x)\}$ and $|\mu(x)|<1/2$.  
 Hence the only pole of $R_\pi(\zeta,x)$ in the considered domain is $\mu(x)$  which has order 1.
 Applying the residue theorem we get
 \begin{align*}\int _{w-in}^{w+in} e^{-\zeta t} &R_\pi(\zeta,x) u\,d\zeta + \int _{w+in}^{1/2+in} e^{-\zeta t} R_\pi(\zeta,x) u\,d\zeta  + \int _{1/2+in}^{1/2-in} e^{-\zeta t} R_\pi(\zeta,x) u\,d\zeta  \\ &+\int _{1/2-in}^{w-in} e^{-\zeta t} R_\pi(\zeta,x) u\,d\zeta= -2\pi i \res _{\zeta =\mu(x)}( e^{-\zeta t} R_ \pi (\zeta ,x)u)\end{align*}
 By Corollary \ref{cor:resofT} (i) we have \[
        \int _{w\pm in}^{1/2 \pm in} e^{-\zeta t} R_\pi(\zeta,x) u\,d\zeta \stackrel{n\to\infty}{\longrightarrow} 0.
       \]

Integration by parts yields
\begin{align*}
 \int _{1/2-in}^{1/2+in}& e^{-\zeta t} R_\pi(\zeta,x) u\,d\zeta=t\inv \int _{1/2-in}^{1/2+in} e^{-\zeta t}\frac{d}{d\zeta} R_\pi(\zeta,x) u\,d\zeta -\left. t\inv e^{-\zeta t}R_\pi(\zeta,x)u\right|_{1/2-in}^{1/2+in}\\
 &=t\inv \int _{1/2-in}^{1/2+in} e^{-\zeta t} R_\pi(\zeta,x)^2 u\,d\zeta-\left. t\inv e^{-\zeta t}R_\pi(\zeta,x)u\right|_{1/2-in}^{1/2+in}.
\end{align*}

Using Corollary \ref{cor:resofT} (i) we furthermore calculate for $|x|\leq r_\pi/2$:
\begin{align*}
&\lim_{n\to\infty} \left\| \int _{1/2-in}^{1/2+in} e^{-\zeta t} R_\pi(\zeta,x) u\,d\zeta\right\| = \lim_{n\to\infty} \left\|t\inv \int_{1/2-in}^{1/2+in}e^{-\zeta t}R_\pi(\zeta,x)^2u\,d\zeta\right\| \\&  \leq t\inv \int_{-\infty}^\infty e^{-t/2} \|R_\pi(1/2 + is,x)\|^2 \,ds\|u\| \leq 4 t\inv e^{-t/2} \int_{-\infty}^\infty \frac 1{1/4 +s^2}\,ds \|u\| \\
&= 8 t\inv e^{-t/2} \int_{-\infty}^\infty \frac 1{1 +s^2}\,ds \|u\|=8\pi t\inv  e^{-t/2}\|u\|.
\end{align*}
In particular, the limit exists.
Notice that 
\begin{align*}
\res&_{\zeta =\mu(x)}( e^{-\zeta t} R_ \pi (\zeta ,x)u) =e^{-\mu(x) t}\res_{\zeta =\mu(x)}(  R_ \pi (\zeta ,x)) u  \\
&=e^{-\mu(x) t}\frac 1 {2\pi i} \int _{|\zeta|=1/2}  R_\pi(\zeta,x)\,d\zeta u= -e^{-\mu(x) t}P(x)u\end{align*}
as the pole has order 1.

Hence,
\[e^{-tT_\pi(x)}u = e^{-\mu(x)t}P(x)u + \frac 1t \frac{1}{2\pi i} \int_{1/2-i\infty}^{1/2+i\infty}e^{-\zeta t}R_\pi(\zeta,x)^2u\,d\zeta \] and the estimate follows from the above calculation.

For the case of a discrete series representation the proof is the same except that we do not collect a residue and  use the estimate of Corollary \ref{cor:resofT} (ii).
\end{proof}

With the decomposition of $L^2(S\M)$ we will now prove Theorem \ref{thm:convergencetoequilibrium}.

\begin{proof}[Proof of Theorem \ref{thm:convergencetoequilibrium}]
 Recall that $L^2(S\M)$ decomposes discretely by Theorem \ref{thm:decomp}. Let $f_\pi$ be the projection of $f$ on $m(\pi)\mc H_\pi$. If $\pi$ is a complementary or principle series representation it corresponds to the eigenvalue $\eta = \frac 14 \lambda_\pi>0$ of $\Delta_\M$. In this case we write $f_\eta$ instead of $f_\pi$. If $\eta=0$ we define $f_\eta$ to be the orthogonal projection of $f$ onto the trivial representation in $L^2(S\M)$.
 With the norm of Section \ref{sec:hyperbolicsurfaces} we have
 \begin{align*}
  C^2\geq&\, \|(-H^2-B^2-\Xi^2)f\|_{L^2(S\M)}^2 \\\geq &\sum_{\eta\in\sigma(\Delta_\M)} \|(-H^2-B^2-\Xi^2)f_\eta\|^2
  +\sum_{\pi =\pi_{\pm2n}^\pm} \|(-H^2-B^2-\Xi^2)f_\pi\|^2\\
  \geq& \sum_{\eta\in\sigma(\Delta_\M)} \|(\eta-2\Xi^2)f_\eta\|^2 \\
  =&  \sum_{\eta\in\sigma(\Delta_\M)} \langle (\eta-2\Xi^2)f_\eta,(\eta-2\Xi^2)f_\eta \rangle\\
  \geq & \sum_{\eta\in\sigma(\Delta_\M)} \eta^2\|f_\eta\|^2
 \end{align*}
 since $-\Xi^2$ is a positive operator.
 
 Thus,
 \begin{align*}
  \sum_ {\eta>C\varepsilon\inv} \|f_\eta\|^2\leq  (C\varepsilon\inv)^{-2} \sum_ {\eta>C\varepsilon\inv} \eta^2\|f_\eta\|^2\leq(C\varepsilon\inv)^{-2} \sum_ {\eta\in\sigma(\Delta_\M)} \eta^2\|f_\eta\|^2\leq \varepsilon^2.
 \end{align*}
Because of $\|e^{-tP_\g}\|\leq 1$ we obtain \begin{equation}
                                             \bigg \|e^{-tP_\g}\sum_ {\eta>C\varepsilon\inv} f_\eta\bigg\|\leq \varepsilon.\label{eq:gleichung}
                                            \end{equation}

We define $\lambda_\eta (\g)\coloneqq \frac{\g^2}2 \mu_\eta(-2\g\inv)$ where $\mu_\eta(x)$ is the eigenvalue of $T_\pi(x)$ obtained in Theorem \ref{thm:specT} and $\pi$ is the representation corresponding to $\eta\in\sigma(\Delta_\M)$ (see Theorem \ref{thm:evofPg}). 
Furthermore, let $\Pi_{\lambda_\eta(\g)}$ be the projection onto the eigenvalue $\lambda_\eta (\g)$ given by $$\Pi_{\lambda_\eta(\g)} f = P_\eta(-2\g\inv)f_\eta\quad \text{with}\quad P_\eta(x)=-\frac 1{2\pi i} \int_{|\zeta|=1/2} R_\pi(\zeta,x)\, d\zeta.$$ 
Note that $\|P_\eta(x)\|\leq 2$ for $|x|\leq (2\sqrt{4\eta+6})\inv$ by Corollary \ref{cor:resofT}. If $\eta=0$ we write  $\lambda_\eta(\g)=0$ and $\Pi_{\lambda_\eta(\g)}f = f_\eta$ and it holds $e^{-tP_\g}f_\eta =f_\eta$.

Then we have by Proposition \ref{prop:semigrouprestriction}  and Theorem \ref{thm:contractionpi} 

\begin{align*}
\bigg \|e^{-tP_\g}f -&\sum_{\stackrel{\eta \in \sigma(\Delta_\M)}{  \eta \leq C{\varepsilon}\inv}}e^{-t \lambda_\eta (\g)}\Pi_{\lambda_\eta(\g)} f \bigg \|_{L^2(S\M)}^2=\sum_{\stackrel{\eta \in \sigma(\Delta_\M)}{  \eta \leq C{\varepsilon}\inv}}\bigg \|e^{-tP_\g}f_\eta -e^{-t \lambda_\eta (\g)}\Pi_{\lambda_\eta(\g)} f_\eta \bigg \|^2\\
& \phantom\leq+ \bigg \|e^{-tP_\g}\sum_ {\eta>C\varepsilon\inv} f_\eta\bigg\|^2 + \sum_{\pi =\pi_{\pm2n}^\pm } \bigg \|e^{-tP_\g}f_\pi\bigg\|^2\\
&\overset{\text{Eq. \eqref{eq:gleichung}}}\leq   \varepsilon^2 + \sum_{\stackrel{\eta \in \sigma(\Delta_\M)}{  \eta \leq C{\varepsilon}\inv}}\bigg \|e^{-t\frac{\g^2}2T_\eta (-2\g\inv)}f_\eta -e^{-t \frac{\g^2}2\mu_\eta(-2\g\inv)}P_\eta (-2\g\inv) f_\eta \bigg \|^2\\
&\phantom= +\sum_{\pi =\pi_{\pm2n}^\pm } \bigg \|e^{-t\frac{\g^2}2T_\pi(-2\g\inv)}f_\pi\bigg\|^2\\
&\overset{\text{Thm. \ref{thm:contractionpi}}}\leq \varepsilon^2 + \sum_{\stackrel{\eta \in \sigma(\Delta_\M)}{  \eta \leq C{\varepsilon}\inv}}\left(\frac 8{t\g^2} e^{-t\g^2/4}\right)^2\|f_\eta\|^2 + \sum_{\pi =\pi_{\pm2n}^\pm } \left(\frac 8{t\g^2} e^{-t\g^2/4}\right)^2\|f_\pi\|^2\\
&\leq \varepsilon^2 +\left(\frac 8{t\g^2} e^{-t\g^2/4}\right)^2 \|f\|_{L^2(S\M)}^2
\end{align*}for every $\g >  \max \{4\sqrt{4C{\varepsilon} \inv +6} , 4\sqrt{32}\}$ where we have used Proposition~\ref{prop:semigrouprestriction} in the first inequality.
\end{proof}
We end this section with the proof of Corollary~\ref{cor:equilibrium}

\begin{proof}
 By Theorem~\ref{thm:evofPg} we have $|\lambda_{\eta}(\g) - \eta|\leq B\inv  $ for all eigenvalues $\eta \leq C\varepsilon\inv$ and $\g > 4B(4C{\varepsilon} \inv +6)^{3/2}$. In particular, $\Re \lambda_{\eta}(\g) \geq \eta_1 - B\inv$. Hence by Theorem~\ref{thm:convergencetoequilibrium},
 \begin{align*}
  \norm{e^{-tP_\g}f - \int_{S\M} f d\mu}_{L^2(S\M)} &\leq \varepsilon + \frac 8 {\g^2t}e^{-\g^2t/4} \|f\|_{L^2(S\M)} \\&\phantom =+ \sum_{\substack{\eta \in \sigma(\Delta_\M) \\ 0 \neq \eta \leq C{\varepsilon}\inv}}\bigg \|e^{-t \lambda_\eta (\g)}\Pi_{\lambda_\eta(\g)} f \bigg \|_{L^2(S\M)}.
 \end{align*}
Furthermore, we have $$\|e^{-t \lambda_\eta (\g)}\Pi_{\lambda_\eta(\g)} f  \|_{L^2(S\M)}\leq 2 e^{-t(\eta_1-B\inv)}\norm{f}_{L^2(S\M)}.$$
and by the Weyl law $$\sup_N \frac {\# \{\eta \in \sigma (\Delta_\M)\mid \eta\leq N\}}N <\infty.$$
This completes the proof.
\end{proof}

\appendix
\section{\texorpdfstring{Proof of Proposition \ref{prop:kbb}}{Proof of Proposition 2.1}}
The proof that $P_\g$ is hypoelliptic with the subelliptic estimate can be found in \cite[Chapter 2.2]{alexis}. There exist vector fields $X_j$ on $S\M$ such that $\Delta_\bbS = -\sum_{j=1}^d X_j^2$ and $\operatorname{div} X_j = 0$ (see \cite[\S 2.2.6]{alexis}). Hence, the $X_j$ as well as $X$ are skew-symmetric with respect to the inner product of $L^2(S\M)$.  It follows that $\Re \langle P_\g f, f\rangle = \sum \frac 12\g^2 \langle X_j f, X_j f\rangle - \g\Re \langle X f , f \rangle\geq 0$, i.e. $P_\g|_{C^\infty}$ is accretive since $\langle X f, f\rangle \in i\R$.

We show that $\operatorname{Ran}(P_\g|_{C^\infty} + I)$ is dense following the proof of \cite[Prop.~5.5]{witten}. Let $f \in \operatorname{Ran}(P_\g|_{C^\infty} + I)^\perp$. Then we have $\langle f, (P_\g + I)u\rangle = 0$ for all $u \in C^\infty$, hence $(P_{-\g} + I)f = 0$ in $\mc D'$. Since $P_{-\g}$ is hypoelliptic, it follows that $f\in C^\infty$ and $0= \sum \frac 12 \g^2\langle X_j f,X_j f\rangle + \langle f,f \rangle - \g\langle Xf ,f\rangle$. Thus $f = 0$. 

We obtain that the closure $\overline {P_\g|_{C^\infty}}$ is maximal-accretive (see e.g. \cite[Thm.~5.4]{witten}).
An operator $A$ on a Hilbert space is maximal-accretive iff it generates a contraction semigroup $e^{-tA}$ (see \cite[p. 241]{reedsimon}). Hence, $\overline {P_\g|_{C^\infty}}$ generates a contraction semigroup $e^{-tP_\g}$. The adjoint semigroup $(e^{-tP_\g})^\ast$ is generated by $(P_\g|_{C^\infty })^\ast$ that is $\frac 12 \g^2 \Delta_\bbS + \g X$ with domain $\{f \in L^2\mid (\frac 12 \g^2 \Delta_\bbS + \g X)f \in L^2\}$ (see \cite[I.5.14 and II.2.5]{engelnagel}). In particular, this operator is maximal-accretive. In analogy we infer that both $\overline{P_\g|_{C^\infty}}$ and $P_\g$ are maximal-accretive and we conclude that they coincide. Similar arguments can be found in \cite{GS14}.

For the positivity of the generated contraction semigroup we have to check if $$\langle (\operatorname{sign}f) P_\g f, u\rangle \geq \langle |f|, (P_\g)^\ast u\rangle$$ for all real $f\in C^\infty$ and a strictly positive subeigenvector $u$ of $(P_\g)^\ast$ (see \cite[C-II Cor. 3.9]{arendt}). Note that $1$ is a strictly positive eigenvector of $(P_\g)^\ast$ and $\frac 12 \Delta_\bbS (x)$ as well as $- X$ generate stochastic Feller processes on $S_x\M$ and $S\M$ respectively (namely the Brownian motion on $S_x\M$ and the  geodesic flow). Hence, $e^{-t\Delta_\bbS(x)}$ and $e^{tX}$ define positive semigroups so that $\langle (\operatorname{sign} f) \Delta_\bbS(x)f, 1\rangle_{S_x\M}\geq 0$ for $f\in C^\infty (S_x\M)$ and $\langle (\operatorname{sign} f) (-X)f, 1\rangle\geq 0$ for $f\in C^\infty(S\M)$ (see \cite[C-II Thm.2.4]{arendt}). Combining both statements completes the proof.

\section{\texorpdfstring{Proof of Proposition \ref{prop:geometricoperators}}{Proof of Proposition 2.4}}
Since $\H \to \M$  is a local diffeomorphism it suffices to consider $S\H$ instead of $S\M$.
 Let $f\in C^\infty(S\H)$. $\g(t)\coloneqq e^t i$ is a geodesic in $\H$ with $\g(0)=i$ and $\dot\g(0)=i\in T_i\H$. Hence, $$\phi_t(i,i)=(e^ti, e^ti)=(a_t i, T_i a_t i)=a_t (i,i) = \exp(tH) (i,i)$$ where $a_t=\operatorname{diag}(e^{t/2},e^{-t/2})=\exp(tH).$ Using that $\isom^+(\H)$ commutes with $\phi_t$ and the identification $G\simeq S\H$ we deduce $\phi_t(x)=x\exp(tH)$ for all $x\in G$. We conclude \[Xf(x)=\d f(\phi_t(x))=\d f(x\exp(tH))=Hf(x).\]

 For the spherical Laplacian let $(z,v)=g.(i,i)=(g. i, T_ig\,i)\in S\H$.
 Then\begin{align*}
 -\Delta_\bbS f(z,v)&=\d\d[s]f(z,v\cdot e^{it}e^{is})&&     =\d\d[s]f(gi,T_ig(ie^{it}e^{is}))\\
 &=\d\d[s]f(g.(i,ie^{it}e^{is}))&
 &=\d\d[s]f(g\exp(t\Xi)\exp(s\Xi)(i,i))\\
 &=\Xi^2f(z,v).&&
 \end{align*}
For the last part we observe that $-H^2-B^2=\frac 14 \Omega - \Xi^2$ is $K$-invariant since $\Omega\in Z(\mc U(\mf g))$ and $\operatorname{Lie}(K)=\R\Xi$. Let $f\in C^\infty(\H)=C^\infty(G/K)=C^\infty(G)^K$, i.e. $f\in C^\infty(S\H)$ and $f$ is constant in each fiber. By abuse of notation we will use the same symbol $f$ in  every  isomorphic space. We calculate
\[Hf(e_G)=\d f(\exp(tH)) = \d f((e^t i,e^t i)) = \d f(e^t i) = \frac{\partial}{\partial y}f (i).\]
Let $B=E+F$ where $E$ is the upper triangular part of $B$ and $F$ the lower triangular part. Then we have 
\[2Ef(e_G)=\d f(\exp(2tE)) = \d f(i+t) = \frac \partial {\partial x} f(i)\]
and 
\[2Ff(e_G)=\d f(\exp(2tF)) = \d f(i(it+1)\inv)= \frac\partial {\partial x} f(i)\] since $\d i(it+1)\inv =1$.
To sum up, \[(-H^2-B^2)f(i)=-\left(\frac\partial {\partial y}\right)^2 - \left(\frac\partial {\partial x}\right)^2f(i) =\Delta_\H f (i).\] As both operators are (left) $G$-invariant they have to coincide.

% \bibliography{literatur}
% \bibliographystyle{amsalpha}

\newcommand{\etalchar}[1]{$^{#1}$}
\providecommand{\bysame}{\leavevmode\hbox to3em{\hrulefill}\thinspace}
\providecommand{\MR}{\relax\ifhmode\unskip\space\fi MR }
\providecommand{\MRhref}[2]{%
  \href{http://www.ams.org/mathscinet-getitem?mr=#1}{#2}
}
\providecommand{\href}[2]{#2}

\end{document}